\documentclass[a4paper]{amsart}
\usepackage{amscd}

\DeclareMathOperator{\tr}{tr}

\newcommand{\Nil}{\mathrm{Nil}}
\newcommand{\abs}[1]{\left\lvert #1\right\rvert}
\newcommand{\bigabs}[1]{\bigl\lvert #1\bigr\rvert}

\newcommand{\norm}[1]{\left\lVert #1\right\rVert}
\newcommand{\bignorm}[1]{\bigl\lVert #1\bigr\rVert}

\renewcommand{\mid}{\;:\;}

\theoremstyle{plain}
\newtheorem{theorem}{Theorem}

\newtheorem{proposition}{Proposition}
\newtheorem{corollary}{Corollary}
\theoremstyle{definition}

\theoremstyle{remark}
\newtheorem*{remark}{Remark}

\thanks{This work was supported by the project FIRB ``Geometria Differenziale e teoria geometrica delle funzioni'',
 the project PRIN
\lq\lq  Variet\`a reali e complesse: geometria, topologia e analisi armonica" and by G.N.S.A.G.A. of I.N.d.A.M}

\begin{document}

\title[The Calabi-Yau equation on the Kodaira-Thurston manifold]{The Calabi-Yau equation on the Kodaira-Thurston manifold,
viewed as an $S^1$-bundle over a $3$-torus}
\author{E.~Buzano, A.~Fino and L.~Vezzoni}

\begin{abstract}
We prove that the Calabi-Yau equation on the  Kodaira-Thurston manifold has a unique solution for every
$S^1$-invariant initial datum.
\end{abstract}

\maketitle

\section{Introduction and statement of the result}
The celebrated Calabi-Yau theorem affirms that given a compact K\"ahler manifold $(M^n,\Omega,J)$ with first Chern class
$c_1(M^n)$, every $(1,1)$-form $\tilde \rho\in 2\pi c_1(M^n)$ is the \emph{ Ricci form} of a unique K\"ahler metric whose
K\"ahler form belongs to the cohomology class $[\Omega]$.  This theorem was conjectured by Calabi in \cite{C} and
subsequently proved by Yau in \cite{Y}. The Calabi-Yau theorem can be alternatively reformulated in terms of symplectic
geometry by saying that, given a compact K\"ahler manifold $(M^n,\Omega,J)$ and a volume form $\sigma$ satisfying the
normalizing condition
\begin{equation*}
\int_{M^n} \sigma=\int_{M^n} \Omega^n,
\end{equation*}
then there exists a unique K\"ahler form $\tilde \Omega$ on $(M^n,J)$ solving
\begin{equation}\label{eqn:105}
\tilde \Omega^n=\sigma,\qquad  [\tilde \Omega]=[\Omega].
\end{equation}
Equation~\eqref{eqn:105} still makes sense in the \emph{almost-K\"ahler} case, when $J$ is merely an almost-complex
structure. In this more general context \eqref{eqn:105} is usually called the \emph{ Calabi-Yau equation.}

In \cite{D} Donaldson described a project about compact symplectic $4$-manifolds involving the Calabi-Yau equation and
showed the uniqueness of the solutions. Donaldson's project is principally based on a conjecture stated in \cite{D}
whose confirmation would lead to new fundamental results in symplectic geometry.  Donaldson's project was partially
confirmed by Taubes in \cite{T2} and strongly motivates the study of the Calabi-Yau equation on non-K\"ahler
$4$-manifolds.

In \cite{W} Weinkove proved that the Calabi-Yau equation can be solved if the torsion of $J$ is sufficiently small and
in \cite{TWY} Tosatti, Weinkove and Yau proved the
 Donaldson  conjecture assuming an extra condition on the curvature and the torsion of the almost-K\"ahler metric.
Furthermore, Tosatti and Weinkove solved in \cite{TW} the Calabi-Yau equation on the Kodaira-Thurston manifold
assuming the initial datum $\sigma$ invariant under the action of a $2$-dimensional torus $T^2$.  The Kodaira-Thurston
is historically the first example of symplectic manifold without K\"ahler structures (see \cite{T3, Ab}) and it is
defined as the direct product of a compact quotient of the $3$-dimensional Heisenberg group by a lattice with the
circle $S^1$.  In \cite{FLSV} it is proved  that when $\sigma$ is $T^2$-invariant, the Calabi-Yau equation on the
Kodaira-Thurston manifold can be reduced to a Monge-Amp\`ere  equation on a torus which has always a solution. Moreover
in \cite{FLSV,BFV}  the same equation is studied in every $T^2$-fibration over a $2$-torus.

The \emph{ Kodaira-Thurston manifold} is defined as the compact $4$-manifold
\begin{equation*}
M=\Nil^3/\Gamma\times S^1,
\end{equation*}
 where
$\Nil^3$ is  the $3$-dimensional  real Heisenberg group
\begin{equation*}
\Nil^3=\left\{\left[\begin{smallmatrix}1&x&z\\0&1&y\\0&0&1\end{smallmatrix}\right]\mid x,y,z \in \mathbb R\right\}
\end{equation*}
and $\Gamma$ is  the lattice in $\Nil^3$ of matrices having integers entries.

Therefore $M$ is parallelizable  and has the global left-invariant co-frame
\begin{equation}\label{eqn:110}
e^1= d y,\quad e^2= dx,\quad e^3= dt,\quad e^4= dz-x dy
\end{equation}
satisfying the structure equations
\begin{equation}
\label{eqn:1}
 de^1= de^2= de^3=0,\quad  de^4=e^{12},
\end{equation}
with
\begin{equation*}
e^{ij}=e^i\wedge e^j.
\end{equation*}

Since $\Nil^3/\Gamma\times S^1=(\Nil^3\times \mathbb R)/(\Gamma\times \mathbb Z)$,  the Kodaira-Thurston manifold $M$
is a $2$-step nilmanifold and every left-invariant almost-K\"ahler structure on $\Nil^3 \times \mathbb R$ projects to
an almost-K\"ahler structure  on $M$. Moreover, the compact $3$-dimensional manifold $N=\Nil^3/\Gamma$ is the total
space of  an $S^1$-bundle over a $2$-dimensional torus $T^2$ with projection $\pi_{xy}\colon N\to T^2_{xy}$ and $M$
inherits a structure of  principal $S^1$-bundle over the $3$-dimensional torus $T^3=T^2_{xy}\times S^1_t $, i.e.
\begin{equation*}
\begin{CD}
S^1 @>>>  N\times S^1=M  \\
@.   @VVV \\
{}   @. T^2 \times S^1=T^3.
\end{CD}
\end{equation*}
Then it makes sense to consider differential forms  invariant by the action of the fiber  $S^1_z$. A $k$-form $\phi$ on $M$ is invariant by the action of the fiber $S^1_z$ if  its
coefficients with respect to the global basis $e^{j_1}\wedge \cdots \wedge e^{j_k}$ do not depend on
the variable $z$.

These observations allow to extend the analysis in \cite{TW,FLSV} from $T^2$-invariant to  $S^1$-invariant data  $\sigma$.

\medskip
Consider on $M$  the canonical metric
\begin{equation}\label{eqn:191}
g=\sum_{k=1}^4 e^k\otimes e^k,
\end{equation}
and the  compatible symplectic form
\begin{equation*}
\Omega=e^{13}+e^{42}.
\end{equation*}
The pair $(\Omega,g)$ specifies an almost-complex structure $J$  making $(\Omega,J)$ an almost-K\"ahler structure. Observe that 
\begin{equation*}
J e^1=e^3\quad\textup{and}\quad J e^4=e^2.
\end{equation*}
Then we can   consider the
 Calabi-Yau equation
\begin{equation}\label{eqn:103}
(\Omega+d\alpha)^2=\mathrm e^F\,\Omega^2,
\end{equation}
where the unknown $\alpha$ is a smooth $1$-form on $M$ such that
\begin{equation}\label{eqn:88}
J(d\alpha)=d\alpha,
\end{equation}
and the datum $F$ is a smooth function on $M$ satisfying
\begin{equation}\label{eqn:97}
\int_M\mathrm e^F\,\Omega^2 = \int_M \Omega^2.
\end{equation}
We have the following
\begin{theorem}\label{thm:12}
The Calabi-Yau equation \eqref{eqn:103} has a unique solution $\tilde\omega=\Omega+d\alpha$
for every  $S^1$-invariant volume form $\sigma=\mathrm e^F\,\Omega^2$
 such that
\begin{equation}\label{eqn:29}
\int_{T^3}\mathrm e^F\,dV=1,
\end{equation}
where $dV$ is the volume form $dx\wedge dy\wedge dt$ on $T^3$.
\end{theorem}

Since uniqueness follows from a general result in \cite{D}, we need only to prove existence. This  will be done in
two steps. First in Section~\ref{sec:1}  we reduce equation~\eqref{eqn:103} to a fully nonlinear PDE on the
$3$-dimensional base torus $T^3$. Then in Section~\ref{sec:5} we show that such an equation is solvable. Section
\ref{sec:4} concerns the a-priori estimates needed in Section~\ref{sec:5}.

\bigskip
 With some minor changes in the proof, it is possible to generalize Theorem \ref{thm:12} to the larger class of invariant
almost-K\"ahler structures on the Kodaira-Thurston manifold.
All positively-oriented invariant almost-K\"ahler structures compatible with the canonical metric \eqref{eqn:191} can be obtained by rotating the  symplectic form $\Omega=e^{13}+e^{42}$. Indeed, since the three forms
\begin{equation*}
\Omega=e^{13}+e^{42}\,,\quad \Omega'=e^{14}+e^{23}\,,\quad \Omega''=e^{12}+e^{34}
\end{equation*}
are a basis of invariant self-dual $2$-forms, every positively-oriented invariant $2$-form  $\omega$ compatible with
$g$ can be written as
\begin{equation*}
\omega=A\Omega+B\Omega'+C\Omega''
\end{equation*}
for some constants $A,B,C$ satisfying $A^2+B^2+C^2=1$. The condition $d\omega=0$ is equivalent to $C=0$ and therefore every positively oriented symplectic $2$-form compatible with $g$ can be written as
 \begin{equation*}
\omega_\theta=(\cos\theta\, e^1+\sin\theta\, e^2)\wedge e^3-(-\sin\theta\, e^1+\cos\theta\, e^2)\wedge e^4,
\end{equation*}
for some $\theta\in [0,2\pi)$.

\begin{theorem}\label{thm:10}
Assume either $\cos\theta=0$ or $\tan\theta\in\mathbb Q$.
Then the Calabi-Yau equation
\begin{equation*}
(\omega_{\theta}+d\alpha)^2={\rm e}^F\omega_\theta^2,\,\quad J_\theta(d\alpha)=0
\end{equation*}
has a unique solution $\tilde\omega=\omega_\theta+d\alpha$ for every  $S^1$-invariant
volume form $\sigma=\mathrm e^F\,\omega_\theta^2$ satisfying \eqref{eqn:29}.
\end{theorem}

In Section \ref{sec:6} we give some details on how to modify the proof of  Theorem~\ref{thm:12} in order to prove Theorem~\ref{thm:10}.

Observe that for $\theta=0$, $\omega_0$ is the form  $\Omega=e^{13}+e^{42}$ considered in Theorem~\ref{thm:12}, while $\omega_{\pi/2}=e^{14}+e^{23}$ is the symplectic form $\Omega'$.

\medskip
\noindent \emph{Acknowledgements}. We  would like to thank Valentino Tosatti  for  useful  remarks and helpful
comments on a preliminary version of the present paper. Moreover, we are grateful to the anonymous referee for useful comments and improvements.

\section{Reduction to a single elliptic equation}\label{sec:1}

The dual frame of \eqref{eqn:110} is
\begin{equation*}
e_1=\partial_y+x\partial_z,\quad e_2=\partial_x,\quad
e_3=\partial_t,\quad e_4=\partial_z.
\end{equation*}
If $u$ is $S^1$-invariant, it does not depend on $z$ and we have
\begin{equation*}
e_1 u=\partial_yu=u_y,\quad e_2 u=\partial_xu=u_x,\quad e_3 u=\partial_tu=u_t,\quad e_4 u=0.
\end{equation*}
It is convenient to set
\begin{equation}\label{eqn:171}
\partial_1=\partial_y,\quad\partial_2=\partial_x,\quad\partial_3=\partial_t,
\end{equation}
so
the differential  can be written as
\begin{equation*}
du=\sum_{i=1}^3\partial_iu\,e^i.
\end{equation*}

\begin{theorem}\label{thm:1}
 Given a smooth function
$u:T^3\to \mathbb R$ such that
\begin{equation}\label{eqn:96}
\int_{T^3} u\,dV =0,
\end{equation}
set
\begin{equation}\label{eqn:72} \alpha=d^{\mathrm c}u-ue^1.
\end{equation}
Then
the $1$-form~\eqref{eqn:72}
satisfies equation~\eqref{eqn:88}. Moreover $\alpha$ solves equation~\eqref{eqn:103}
if and only if $u$ is a solution to the fully non-linear PDE
\begin{equation}\label{eqn:8}
(u_{xx}+1)(u_{yy}+u_{tt}+u_t+1)-u_{xy}^2-u_{xt}^2={\mathrm e}^F.
\end{equation}
 \end{theorem}
\begin{proof}
Thanks to  \eqref{eqn:1} we have
\begin{align*}  dd^{\mathrm c} u&=\sum_{i=1}^3\sum_{j=1}^3\partial_i\partial_ju\,e^i\wedge Je^j-\partial_2u\,e^{12}
\\&=\sum_{i=1}^3\sum_{j=1}^3\partial_i\partial_j u\,e^i\wedge Je^j+d(ue^1)+\partial_3u\,e^{13}.
\end{align*}
Therefore $d\alpha$ is of type $(1,1)$ and
\begin{equation*}
d\alpha\begin{aligned}[t]&=\sum_{i=1}^3\sum_{j=1}^3\partial_i\partial_j u\,e^i\wedge Je^j+\partial_3 u\,e^{13}
\\&=(u_{yy}+u_{tt}+u_t)e^{13}-u_{xx}e^{24}+u_{xy}(e^{23}-e^{14})+u_{xt}(e^{12}-e^{34}).\end{aligned}
\end{equation*}
Then a simple computation shows that  $\alpha$  satisfies \eqref{eqn:103} if and only if $u$  satisfies \eqref{eqn:8}.
\end{proof}

We end this section by proving ellipticity of equation~\eqref{eqn:8}.

First we fix some notation.
 Functions on the $3$-torus can be
identified with functions $u:\mathbb R^3\to\mathbb R$ which are $1$-periodic in each variable.

For any non-negative integer $n$, we denote by $C^{n}(T^3)$  the Banach space of $C^n$ functions $u\colon T^3\to
\mathbb R$ equipped with norm
\begin{equation*}
\norm{ u}_{C^n}=\max_{m\le n}\abs{ u}_{C^m},
\end{equation*}
where\,\footnote{\ We employ Schwartz multi-index notation:
$\partial^\kappa=\partial^{\kappa_1}_1\partial^{\kappa_2}_2\partial^{\kappa_3}_3$ and
$\abs{\kappa}=\kappa_1+\kappa_2+\kappa_3$.}
\begin{equation*}
 \abs{ u}_{C^m}=\max_{\abs{\kappa}=m}\sup_{q\in\mathbb R^3}\bigabs{ \partial^\kappa u(q)}.
\end{equation*}

Given $0<\rho< 1$ and $u\in C^0(T^3)$, we set
\begin{equation*}
\bigl[\!\!\bigl[u(q)\bigr]\!\!\bigr]_\rho=\sup_{\scriptscriptstyle 0<\abs{h}\le 1}
\bigabs{u(q+h)-u(q)}\,\abs{h}^{-\rho}.
\end{equation*}
For every non-negative integer  $n$ and real number $0<\rho<1$, define the space $C^{n+\rho}(T^3)$ of  functions $u\in
C^n(T^3)$ such that
\begin{equation*}
\abs{ u}_{C^{n+\rho}}=
\max_{\abs{\kappa}= n}\sup_{q\in\mathbb R^3}\bigl[\!\!\bigl[\partial^\kappa u(q)\bigr]\!\!\bigr]_\rho
<\infty.
\end{equation*}
$C^{n+\rho}(T^3)$ is a Banach space with respect to the norm
\begin{equation*}
\norm{ u}_{C^{n+\rho}}=\max\Bigl\{\norm{ u}_{C^n},\abs{ u}_{C^{n+\rho}}\Bigr\}.
\end{equation*}
In conclusion we have defined  $C^{\sigma}(T^3)$ for every non negative real number $\sigma$.

Finally, we denote by $\tilde C^{\sigma}(T^3)$ the closed subspace of all $u\in C^\sigma(T^3)$ satisfying
\begin{equation*}
\int_{T^3}u\, dV  =0.
\end{equation*}

\begin{proposition}\label{pro:1}
Let $u\in \tilde C^2(T^3)$ be a solution to \eqref{eqn:8}. Then we have
\begin{equation}
\label{eqn:9}
u_{xx}> -1
\end{equation}
and
\begin{equation}\label{eqn:34}
u_{yy}+u_{tt}+u_t> -1.
\end{equation}
\end{proposition}
\begin{proof}
Indeed, from equation~\eqref{eqn:8}, we have
\begin{equation*}
(u_{yy}+u_{tt}+u_t+1)(u_{xx}+1)\ge \mathrm e^F>0.
\end{equation*}
This implies that $u_{yy}+u_{tt}+u_t+1$ and $u_{xx}+1$ have always the same sign. But at a point where $u$ attains its
minimum we must have
\begin{equation*}
u_{xx}+1\ge 1. \qedhere
\end{equation*}
\end{proof}

Let
\begin{equation*}
\Delta u=u_{xx}+u_{yy}+u_{tt},
\end{equation*}
be the standard Laplacian in $\mathbb R^3$.

Now we prove ellipticity of equation \eqref{eqn:8}.
 \begin{proposition}\label{pro:5} Let
$u\in \tilde C^2(T^3)$ be a solution to  equation~\eqref{eqn:8}. Then we have
\begin{equation}\label{eqn:36}
0<2\mathrm e^{F/2}\le \Delta u+u_t+2,
\end{equation}
and
\begin{multline}\label{eqn:168}
(u_{xx}+1)(\eta^2+\tau^2)
+(u_{yy}+u_{tt}+u_t+1)\xi^2-
2u_{xy}\xi\eta-2u_{xt}\xi\tau\ge
\\ \ge \Lambda(u)\bigl(\xi^2+\eta^2+\tau^2\bigr),\qquad\textup{for all $(\xi,\eta,\tau)\in\mathbb R^3$},
\end{multline}
where
\begin{equation}\label{eqn:61}
\Lambda(u)=\frac 12\Bigl(\Delta u+u_t+2-\sqrt{(\Delta u+u_t+2)^2-4\mathrm e^F}\Bigr).
\end{equation}
\end{proposition}
\begin{remark}
The left-hand-side of \eqref{eqn:168}  is the principal symbol of the linearization of \eqref{eqn:8} at the solution $u$.
Since a non-linear equation is elliptic on a set $S$ if its linearization at any $u\in S$ is elliptic, we have that equation \eqref{eqn:8} is elliptic on the set of all of its solutions $u\in \tilde C^2(T^3)$.
\end{remark}
\begin{proof}
Inequality \eqref{eqn:36} follows from  \eqref{eqn:9}, \eqref{eqn:34} and \eqref{eqn:8}.

A simple computation shows that the characteristic polynomial of the matrix
\begin{equation*}
P(u)=\begin{bmatrix}
u_{yy}+u_{tt}+u_t+1 & u_{xy} & u_{xt} \\
u_{xy} & u_{xx}+1 & 0 \\
u_{xt} & 0 & u_{xx}+1
\end{bmatrix}
\end{equation*}
associated to  the quadratic form at left-hand side
of \eqref{eqn:168} is
\begin{equation*}
\bigl(\lambda-(u_{xx}+1)\bigr)\bigl(\lambda^2-(\Delta u+u_t+2)\lambda+\mathrm e^F\bigr).
\end{equation*}
Then the eigenvalues of $P(u)$ are \begin{equation*}\lambda_\pm=\frac 12\Bigl(\Delta u+u_t+2\pm\sqrt{(\Delta u+u_t+2)^2-4\mathrm
e^F}\Bigr)\end{equation*} and $u_{xx}+1$. Since
\begin{align*}
(\Delta u+u_t+2)^2-4\mathrm e^F&=\bigl((u_{yy}+u_{tt}+u_t+1)-(u_{xx}+1)\bigr)^2+u_{xy}^2+u_{xt}^2
\\&\ge
\bigl((\Delta u +u_t+2)-2(u_{xx}+1)\bigr)^2,
\end{align*}
we have
\begin{equation*}
\lambda_-\le u_{xx}+1\le \lambda_+,
\end{equation*}
and the proof is complete.
\end{proof}

\section{A priori estimates}\label{sec:4}
\subsection{$C^0$-estimate}

\begin{proposition}\label{pro:10}
We have
\begin{equation}
\label{eqn:10}
\bigabs{u_x}\le 1,
\end{equation}
for all  solution $u$ to  \eqref{eqn:8}.
\end{proposition}
\begin{proof}
 Fix
$(x,y,t)\in\mathbb R^3$,
 and consider the periodic function
\begin{equation*}
v(s)=u(x+s,y,t).
\end{equation*}
 We have
\begin{equation*}
v''(s)=u_{xx}(x+s,y,t)\ge -1.
\end{equation*} Let $s_0\in[0,1]$ be a critical point of $v$. Then we have
\begin{equation*}
v'(s)=\int_{s_0}^s v''(r)\,r\begin{cases}\ge -(s-s_0)\ge -1,& s_0\le s\le s_0+1,\\
\le -(s-s_0)\le 1,& s_0-1\le s\le s_0.\end{cases}
\end{equation*}
By periodicity  we get that these estimates hold everywhere, in particular we obtain
\begin{equation*} \abs{u_x(x,y,t)}= \abs{v'(0)}\le 1.
\qedhere
\end{equation*}
\end{proof}

Denote by
\begin{equation*}
\nabla u=\begin{bmatrix}u_x \\ u_y \\ u_t\end{bmatrix}
\end{equation*}
the standard gradient of $u$. We have
\begin{equation*}
\abs{\nabla u}^2=u_x^2+u_y^2+u_t^2
\end{equation*}
thus, if we  set
\begin{equation*}
\abs{\nabla u}_{C^0}=\bigabs{\abs{\nabla u}}_{C^0},
\end{equation*}
 we have
\begin{equation*}
\abs{u}_{C^1}\le \abs{\nabla u}_{C^0}\le \sqrt 3\, \abs{u}_{C^1}.
\end{equation*}

In this paper all $L^p$ norms are taken on the torus $T^3$. In particular  we set
\begin{equation*}
\norm{\nabla u}_{L^2}^2=\int_{T^3}\abs{\nabla u}^2\,dV=\int_{T^3}(u_x^2+u_y^2+u_t^2)\,dV.
\end{equation*}

\begin{theorem} Given a real number $p\ge 2$,   we have
\begin{equation}
\label{eqn:11}
\bignorm{\nabla  \abs{u}^{p/2}}_{L^2}^2\le  \frac {p^2}{16}\, \norm{u}_{L^p}^p+
\frac {5p^3}{16}\,\abs{1+\mathrm e^F}_{C^0}\norm{u}_{L^p}^{p-1},
\end{equation}
for all $u\in \tilde C^2(T^3)$ satisfying equation~\eqref{eqn:8}.
\end{theorem}
\begin{proof}
From Theorem~\ref{thm:1} we have that
\begin{equation}\label{eqn:70}
\alpha=d^cu-ue^1
\end{equation}
solves equation~\eqref{eqn:103}, which can be re-written as
\begin{equation*}
(\mathrm e^F-1)\,\Omega^2= d\alpha\wedge (\Omega+ \tilde \Omega),
\end{equation*}
where
\begin{equation*}
\tilde \Omega=\Omega+ d\alpha.
\end{equation*}
Since
\begin{align*}
 d\bigl(u\abs{u}^{p-2}\bigr)&=\abs{u}^{p-2}\, du+u\,(p-2)\abs{u}^{p-3}\frac u{\abs u}\, du
 \\&=(p-1)\abs{u}^{p-2}\, du,\qquad \textup{for $u\ne 0$},
\end{align*}
we have
\begin{multline}
\label{eqn:12}
\int_{T^{3}} d\Bigl(\bigl(u\abs{u}^{p-2}\alpha\bigr)\wedge (\Omega+\tilde \Omega)\Bigr)
=
\\=(p-1)\int_{T^{3}}\abs{u}^{p-2}\, du\wedge \alpha\wedge (\Omega+\tilde \Omega)+
\int_{T^{3}}\abs{u}^{p-2} u (\mathrm e^F-1)\,\Omega^2
\end{multline}
and Stokes' theorem implies
\begin{equation}
\label{eqn:13}
\int_{T^{3}}\abs{u}^{p-2}\, du\wedge \alpha\wedge (\Omega+\tilde \Omega)=
\frac {1}{p-1}\int_{T^{3}}(1-\mathrm e^F)\abs{u}^{p-2} u \,\Omega^2.
\end{equation}
Taking into account that
\begin{equation}
 \label{eqn:71}
\begin{aligned}[t]\tilde \Omega =&(u_{yy}+u_{tt}+u_t+1)e^{13}-(u_{xx}+1)e^{24},
\\&+u_{xy}(e^{23}-e^{14})+u_{xt}(e^{12}-e^{34}),\end{aligned}
\end{equation}
we have
\begin{equation}
\label{eqn:14}
 d u\wedge \alpha \wedge \Omega=\frac 12\,\Bigl(u_x^2+u_y^2+u_t(u_t+u)\Bigr)\,\Omega^2,
\end{equation}
and
\begin{equation}
\label{eqn:15}
\begin{aligned}[t] d u\wedge\alpha\wedge \tilde \Omega =&\frac 12\,
\biggl(u_y^2+\Bigl(u_t+\frac 12\,u\Bigr)^2\biggr)(u_{xx}+1)\Omega^2
\\&+\frac 12\,
u_x^2(u_{yy}+u_{tt}+u_t+1)\Omega^2
\\
&-\biggl(u_xu_yu_{xy}+u_x\Bigl(u_t+\frac 12\,u\Bigr)u_{xt}\biggr)\Omega^2
\\&-\frac 18\,u^2(u_{xx}+1)\Omega^2.
\end{aligned}
\end{equation}
Thanks to
\eqref{eqn:168}, we obtain from  \eqref{eqn:15} that
\begin{equation*}
d u\wedge\alpha\wedge \tilde \Omega
\ge-\frac 18\, u^2(u_{xx}+1)\Omega^2.
\end{equation*}

 Then from \eqref{eqn:13} and \eqref{eqn:14} we get
\begin{multline}
\label{eqn:16}
\int_{T^{3}}\abs{u}^{p-2}
\Bigl(u_x^2+u_y^2+u_t(u_t+u)\Bigr)\, dV
\le
\\
\le \frac 14\, \int_{T^{3}} \abs{u}^p(u_{xx}+1)\, dV+
\frac {2}{p-1}\int_{T^{3}} (1-\mathrm e^F)\abs{u}^{p-2} u\, dV  .
\end{multline}
An integration by parts gives
\begin{equation*}
\int_{T^{3}}\abs{u}^{p-2}uu_t\, dV
=(1-p)\int_{T^{3}}\abs{u}^{p-2}uu_t\, dV,
\end{equation*}
therefore we have
\begin{equation*}
\int_{T^{3}}\abs{u}^{p-2}uu_t\, dV=0.
\end{equation*}
Since, moreover
\begin{equation*}
\int_{T^{3}} \abs{u}^pu_{xx}\, dV  =-p\int_{T^{3}} \abs{u}^{p-2}uu_x^2\, dV  ,
\end{equation*}
estimates \eqref{eqn:10} and  \eqref{eqn:16} imply
\begin{multline}
\label{eqn:17}
\int_{T^{3}} \abs{u}^{p-2}\abs{\nabla  u}^2\, dV  \le
\frac 14\int_{T^{3}} \abs{u}^p\, dV  +
\\+\Bigl(\frac {p}4+\frac {2}{p-1}\,
\abs{1-\mathrm e^F}_{C^0}\Bigr)\int_{T^{3}} \abs{u}^{p-1}\, dV.
\end{multline}
But the left-hand side can be rewritten as
\begin{equation*}
\int_{T^{3}} \abs{u}^{p-2}\abs{\nabla  u}^2\, dV
 =\frac 4{p^2}\int_{T^{3}} \bigabs{\nabla \abs{u}^{p/2}}^2\, dV.
\end{equation*}
Moreover
\begin{equation*}
\frac {p}4+\frac {2}{p-1}\,
\abs{1-\mathrm e^F}_{C^0}\le \frac {5p}4\,\abs{1+\mathrm e^F}_{C^0},\qquad \textup{for $p\ge 2$},
\end{equation*}
then \eqref{eqn:17} becomes
\begin{equation}
\label{eqn:18}
\int_{T^{3}} \bigabs{\nabla \abs{u}^{p/2}}^2\, dV \le
 \frac {p^2}{16}\, \int_{T^{3}} \abs{u}^p\, dV+
\frac {5p^3}{16}\abs{1+\mathrm e^F}_{C^0}\int_{T^{3}} \abs{u}^{p-1}\, dV.
\end{equation}
Since $T^3$ has measure $1$,  we have
\begin{equation}
\label{eqn:19}
\norm{u}_{L^{p-1}}\le \norm{u}_{L^p}.
\end{equation}
Estimate  \eqref{eqn:11} follows from \eqref{eqn:18} and \eqref{eqn:19}.
\end{proof}
It is rather natural to compare estimate \eqref{eqn:11} with the  classical a priori Yau's estimate
\begin{equation*}
\bignorm{\nabla  \abs{\varphi}^{p/2}}_{L^2}^2\le
\frac {mp^2}{4p-1}\,\left(\abs{1-\mathrm e^F}_{C^0}\right)\norm{\varphi}_{L^p}^{p-1}
\end{equation*}
involving the solutions $\varphi$ to the complex Monge-Amp\`ere equation $(\omega+dd^c\varphi)^m={\rm e}^F\,\omega^m$
in $2m$-dimensional K\"ahler manifolds (see for instance \cite[Proposition 5.4.1]{J}). The right-end side of
\eqref{eqn:11} contains the extra term $\frac {p^2}{16} \norm{u}_{L^p}^p$ due to the presence of $-ue^1$
in \eqref{eqn:72}. This is a problem in the first step of $C^0$-estimate, that is with $p=2$. We take care of this in
the next  proposition.

\medskip
From Strong Maximum Principle $\Delta u$ constant implies $u$ constant, then $-\Delta$  is an operator from $\tilde C^2(T^3)$ into $\tilde C^0(T^3)$. As such its first eigenvalue is $4\pi^2$. This implies the
inequality
\begin{equation}
\label{eqn:20}
4\pi^2\norm{u}_{L^2}^2\le \int_{T^3}-\Delta u\, u\,dV=
\norm{\nabla  u}_{L^2}^2, \qquad \textup{for all $u\in \tilde C^2(T^3)$}.
\end{equation}

\begin{proposition}\label{pro:4}
We have
\begin{equation}\label{eqn:52}
\norm{u}_{L^2}\le
\abs{1+\mathrm e^F}_{C^0},
\end{equation}
for all $u\in \tilde C^2(T^3)$ satisfying equation~\eqref{eqn:8}.
\end{proposition}
\begin{proof}
Since
\begin{equation*}
\bignorm{\nabla \abs{u}}^2_{L^2}=\norm{\nabla  u}^2_{L^2},
\end{equation*}
from \eqref{eqn:11} with $p=2$ and \eqref{eqn:20} we obtain
\begin{equation*}
4\pi^2 \norm{u}_{L^2}^2\le\frac {1}{4}\norm{u}_{L^2}^2+\frac 52\,\abs{1+\mathrm e^F}_{C^0}
\norm{u}_{L^2},
\end{equation*}
which implies \eqref{eqn:52}.
\end{proof}

Now we are ready to prove an  a priori $C^0$ estimate for the solutions to \eqref{eqn:8}:
\begin{theorem}\label{thm:7}
Given  $F\in C^2(T^3)$ satisfying condition \eqref{eqn:29}, there exists a positive constant $C_0$, depending only on
$\abs{F}_{C^0}$  such that
\begin{equation}
\label{eqn:21}
\abs{u}_{C^0}\le C_0,
\end{equation}
for all $u\in \tilde C^2(T^3)$ satisfying equation~\eqref{eqn:8}.
\end{theorem}
\begin{proof}
From Sobolev Imbedding Theorem (see  for instance \cite[Theorem 5.4]{A}), there exists a positive
constant $K$ such that
\begin{equation}
\label{eqn:22}
\norm{w}_{L^6}^2\le K\Bigl(\norm{w}_{L^2}^2+\norm{\nabla  w}_{L^2}^2\Bigr),
\end{equation}
for all $w$ in the Sobolev space $W^{1,2}(T^3)$.

Then from \eqref{eqn:11} and \eqref{eqn:22} we have
\begin{equation}
\label{eqn:23}
\norm{u}_{L^{3p}}^p\begin{aligned}[t]
&\le K\Bigl(1+\frac {p^2}{16}\Bigr) \norm{u}_{L^p}^p+
K\frac {5p^3}{16}\,\abs{1+\mathrm e^F}_{C^0}\norm{u}_{L^p}^{p-1}
\\&\le K\,p^3\norm{u}_{L^p}^p\Bigl(1+\abs{1+\mathrm e^F}_{C^0}\norm{u}_{L^2}^{-1}\Bigr),
\qquad \textup{for all  $p\ge 2$.}
\end{aligned}
\end{equation}
It follows that
\begin{equation*}
\frac{\norm{u}_{L^{3p_k}}}{\norm{u}_{L^{p_k}}}\le
(Mp_k^3)^{1/p_k},\qquad \textup{for all $k\in \mathbb Z_+$},
\end{equation*}
with
\begin{equation}
\label{eqn:24}
M=K\,\Bigl(1+\abs{1+\mathrm e^F}_{C^0}\norm{u}_{L^2}^{-1}\Bigr)
\end{equation}
and
\begin{equation*}
p_k=2\cdot 3^k.
\end{equation*}
Then
\begin{equation*}
\frac{\norm{u}_{L^{3p_n}}}{\norm{u}_{L^2}}\le \prod_{k=0}^n(Mp_k^3)^{1/p_k},\qquad \textup{for all $n\in \mathbb Z_+$}.
\end{equation*}
But
\begin{equation*}
\prod_{k=0}^\infty (Mp_k^3)^{1/p_k}=\exp\biggl(\sum_{k=0}^\infty\frac 1{2\cdot 3^k}\Bigl(\log (8M)+3k\log 3\Bigr)\biggr)=
(8M)^{3/4}3^{3\mu/2},
\end{equation*}
with
\begin{equation*}
\mu=\sum_{k=1}^\infty \frac k{3^k}<\infty.
\end{equation*}
Then
\begin{equation}
\label{eqn:25}
\abs{u}_{C^0}=\sup_{n\in\mathbb N}\norm{u}_{L^{p_n}}\le (8M)^{3/4}3^{3\mu/2}\norm{u}_{L^2}.
\end{equation}
Now from \eqref{eqn:24} and \eqref{eqn:52} we have
\begin{align*}
M^{3/4}\norm{u}_{L^2}&=K^{3/4}\, \Bigl(\norm{u}_{L^2}+
\abs{1+\mathrm e^F}_{C^0}\Bigr)^{3/4}\norm{u}_{L^2}^{1/4}
\\&\le
(2K)^{3/4}\,
\abs{1+\mathrm e^F}_{C^0},
\end{align*}
and \eqref{eqn:21} follows from
 \eqref{eqn:25}.
\end{proof}

\subsection{Estimate of gradient and Laplacian.}
\rule{0pt}{0pt}

We make use of the tensor product notation.
In particular $(\nabla \otimes \nabla )u$ is the Hessian matrix of $u$, and $\tr (\nabla \otimes\nabla)=\Delta$ is the
Laplacian.

Observe that
\begin{equation*}
(\nabla \otimes\nabla)(uv)=v\,(\nabla \otimes\nabla)u
+u\,(\nabla \otimes\nabla)v+(\nabla u\otimes\nabla v)+(\nabla v\otimes\nabla u).
\end{equation*}

\begin{theorem}\label{thm:2}
Given  $F\in C^2(T^3)$ satisfying condition \eqref{eqn:29}, there exists a positive constant $C_1$, depending only on
$\norm{F}_{C^2}$, such that
\begin{equation}\label{eqn:50} \abs{\Delta u}_{C^0} \le
C_1\bigl(1+\abs{u}_{C^1}\bigr),
\end{equation}
for all  $u\in \tilde C^4(T^3)$ solution to equation~\eqref{eqn:8}.
\end{theorem}
\begin{proof}
From equation~\eqref{eqn:8} we obtain
\begin{multline}\label{eqn:31}
\bigl(\Delta F+\abs{\nabla  F}^2+ F_t\bigr)\mathrm e^F= \\ =
\begin{aligned}[t]&(u_{yy}+u_{tt}+u_t+1)(\Delta u_{xx}+u_{xxt})
\\&+(u_{xx}+1)(\Delta u_{yy}+u_{yyt}+\Delta u_{tt}+u_{ttt})
\\
&+(u_{xx}+1)(\Delta u_t+u_{tt})+
2\nabla  u_{xx}\cdot \nabla  (u_{yy}+u_{tt}+u_t)
\\
&-2 u_{xy}(\Delta u_{xy}+u_{xyt})
-2\abs{\nabla  u_{xy}}^2-2 u_{xt}(\Delta u_{xt}+u_{xtt})-2\abs{\nabla  u_{xt}}^2.
\end{aligned}\end{multline}

Consider
\begin{equation}\label{eqn:30}
\Phi=(\Delta u+u_t+2)\mathrm e^{-\mu u},
\end{equation}
where
\begin{equation}\label{eqn:47}
\mu=\frac \epsilon{\max (\Delta u+u_t+2)}
\end{equation}
and $0<\epsilon< 1$ is a constant to be chosen later. Differentiating \eqref{eqn:30} yields
\begin{equation*}
\nabla  \Phi=\mathrm e^{-\mu u}\Bigl(\nabla  (\Delta u+u_t)-\mu (\Delta u+u_t+2)\nabla  u\Bigr),
\end{equation*}
and
\begin{equation*}\begin{aligned}[t]
(\nabla  \otimes \nabla ) \Phi=&
-\mu\mathrm e^{-\mu u}
\Bigl(\nabla  u\otimes\nabla (\Delta u+u_t)+\nabla (\Delta u+u_t)\otimes \nabla  u\Bigr)
\\&
+\mu^2\mathrm e^{-\mu u} \Bigl((\Delta u+u_t+2)\nabla  u\otimes\nabla  u\Bigr)+
\\
&+
\mathrm e^{-\mu u}\Bigl((\nabla \otimes\nabla )(\Delta u+u_t)-\mu(\Delta u+u_t+2)(\nabla  \otimes \nabla ) u\Bigr).\end{aligned}
\end{equation*}

Consider now a point $(x_0,y_0,t_0)$, where  $\Phi$ attains its maximum value.

We have $\nabla \Phi=0$ and  $(\nabla \otimes
\nabla )\Phi\le 0$, so that
\begin{equation}\label{eqn:32} \nabla  (\Delta u+u_t)=\mu (\Delta u+u_t+2)\nabla  u,
\end{equation}
and
\begin{equation}\label{eqn:140}
(\nabla  \otimes \nabla ) (\Delta u+u_t)\le \mu (\Delta u+u_t+2)\Bigl((\nabla  \otimes \nabla ) u
+\mu \nabla  u \otimes   \nabla  u\Bigr).
\end{equation}
In particular, we obtain
\begin{multline}\label{eqn:141}
\Bigl(\mu(\Delta u+u_t+2)(u_{xy}+\mu u_x u_y)-(\Delta u_{xy}+u_{xyt})\Bigr)^2\le \\
\le
\Bigl(\mu(\Delta u+u_t+2)(u_{xx}+\mu u_x^2)-(\Delta u_{xx}+u_{xxt})\Bigr)\cdot\\ \cdot
\Bigl(\mu(\Delta u+u_t+2)(u_{yy}+\mu u_y^2)-(\Delta u_{yy}+u_{yyt})\Bigr),
\end{multline}
and
\begin{multline}\label{eqn:142}
\Bigl(\mu(\Delta u+u_t+2)(u_{xt}+\mu u_x u_t)-(\Delta u_{xt}+u_{xtt})\Bigr)^2\le \\ \le
\Bigl(\mu(\Delta u+u_t+2)(u_{xx}+\mu u_x^2)-(\Delta u_{xx}+u_{xxt})\Bigr)\cdot \\ \cdot
\Bigl(\mu(\Delta u+u_t+2)(u_{tt}+\mu u_t^2)-(\Delta u_{tt}+u_{ttt})\Bigr).
\end{multline}

From  \eqref{eqn:140} we have in particular that
\begin{equation*} \mu (\Delta+u_t+2)(\partial_i\partial_ju+\mu \partial_i u\partial_j u)-(\Delta \partial_i\partial_ju+\partial_t\partial_i\partial_ju)\ge 0, \end{equation*}
for all $1\le i,j\le 3$. Then, form \eqref{eqn:141}, \eqref{eqn:142} and \eqref{eqn:168} with
\begin{equation*}
\begin{cases}
\xi=\Bigl(\mu (\Delta u+u_t+2)(u_{xx}+\mu u_x^2)-(\Delta u_{xx}+u_{xxt})\Bigr)^{1/2},\\
\eta=\Bigl(\mu (\Delta u+u_t+2)(u_{yy}+\mu u_y^2)-(\Delta u_{yy}+u_{yyt})\Bigr)^{1/2},\\
\tau=\Bigl(\mu (\Delta u+u_t+2)(u_{tt}+\mu u_t^2)-(\Delta u_{tt}+u_{ttt})\Bigr)^{1/2},
\end{cases}\end{equation*}
we obtain
\begin{multline}\label{eqn:33}
\begin{aligned}[t]&(u_{yy}+u_{tt}+u_t+1)(\Delta u_{xx}+u_{xxt})
\\&
+(u_{xx}+1)(\Delta u_{yy}+u_{yyt}+\Delta u_{tt}+u_{ttt})
\\
&-2 u_{xy}(\Delta u_{xy}+u_{xyt})
-2 u_{xt}(\Delta u_{xt}+u_{xtt})\le\end{aligned}
\\
\begin{aligned}[t]
\le&\mu(\Delta u+u_t+2)(u_{yy}+u_{tt}+u_t+1)(u_{xx}+\mu u_x^2)
\\&
+\mu(\Delta u+u_t+2) (u_{xx}+1)(u_{yy}+\mu u_y^2+u_{tt}+\mu u_t^2)
\\
&-2\,\mu(\Delta u+u_t+2)\Bigl( u_{xy}(u_{xy}+\mu u_x u_y)
+u_{xt}(u_{xt}+\mu u_x u_t)\Bigr).\end{aligned}
\end{multline}

 Substituting \eqref{eqn:32} and \eqref{eqn:33} into \eqref{eqn:31},
and using \eqref{eqn:36}, we  get
\begin{multline}\label{eqn:40}
\bigl(\Delta F+\abs{\nabla  F}^2+ F_t\bigr)\mathrm e^F\le \\
\le\begin{aligned}[t] &\mu(\Delta u+u_t+2)
(u_{yy}+u_{tt}+u_t+1)(u_{xx}+\mu u_x^2)
\\
&+\mu(\Delta u+u_t+2)(u_{xx}+1)\bigl(u_{yy}+u_{tt}+\mu(u_y^2+u_t^2)\bigr)
\\
&+\mu(\Delta u+u_t+2)(u_{xx}+1)u_t+
2\nabla  u_{xx}\cdot \nabla  (u_{yy}+u_{tt}+u_t)
\\
&-2 \mu(\Delta u+u_t+2)\Bigl(u_{xy}(u_{xy}+\mu u_x u_y)
+u_{xt}(u_{xt}+\mu u_x u_t)\Bigr).\end{aligned}
\end{multline}
On the other side, from \eqref{eqn:32} we have
\begin{multline}\label{eqn:39}
\mu^2(\Delta u+u_t+2)^2\abs{\nabla  u}^2= \abs{\nabla  (\Delta u+u_t)}^2=\\
\begin{aligned}[t]&=\abs{\nabla  u_{xx}}^2+\abs{\nabla  (u_{yy}+u_{tt}+u_t)}^2+
2\nabla  u_{xx}\cdot \nabla  (u_{yy}+u_{tt}+u_t) \\
&\ge 2\nabla  u_{xx}\cdot \nabla  (u_{yy}+u_{tt}+u_t).\end{aligned}
\end{multline}
Eventually from \eqref{eqn:40},  and \eqref{eqn:39} we obtain
\begin{multline}\label{eqn:41}
\bigl(\Delta F+\abs{\nabla  F}^2+ F_t\bigr)\mathrm e^F\le \\
\begin{aligned}[t]\le& \mu(\Delta u+u_t+2)
\Bigl((u_{yy}+u_{tt}+u_t+1)u_{xx}+(u_{xx}+1)(u_{yy}+u_{tt}+u_t)\Bigr)
\\
&-2\,\mu(\Delta u+u_t+2)(u_{xy}^2+u_{xt}^2)
\\
&+2\,\mu^2(\Delta u+u_t+2)\Bigl((u_{yy}+u_{tt}+u_t+1)u_x^2+(u_{xx}+1)(u_y^2+u_t^2)\Bigr)
\\
&+ \mu^2(\Delta u+u_t+2)^2\abs{\nabla  u}^2
\\
\le&
2\mu(\Delta u+u_t+2)\mathrm e^F-
\mu(\Delta u+u_t+2)^2+\mu^2(\Delta u+u_t+2)^2\abs{\nabla  u}^2.
\end{aligned}
\end{multline}

Set
\begin{equation*}
M=\Delta u(x_0,y_0,t_0)+u_t(x_0,y_0,t_0)+2
\end{equation*}
and
\begin{equation*}
u_0=u(x_0,y_0,t_0),
\end{equation*}
so that
\begin{equation*}
\max \Phi=M\mathrm e^{-\mu u_0}.
\end{equation*}
From \eqref{eqn:41} we get
\begin{equation}\label{eqn:45}
\mu M^2\le \bigabs{(\Delta F+F_t)\mathrm e^F}_{C^0}+2\mu M \abs{\mathrm e^F}_{C^0}+
\mu^2 M^2 \abs{\nabla  u}_{C^0}^2.
\end{equation}

Denote by $\tilde u$  the value of $u$ at a point where  $\Delta u+u_t+2$  attains its maximum value. Then, thanks to Theorem \ref{thm:7}, we have
\begin{equation}\label{eqn:46}
M\le\max (\Delta u+u_t+2)\le M \mathrm e^{\mu (\tilde u-u_0)}\le M \mathrm e^{2\mu C_0}.
\end{equation}
Moreover, \eqref{eqn:47} and \eqref{eqn:36} imply
\begin{equation*}
2\mu= \frac {2\epsilon}{\max (\Delta u+u_t+2)}\le \epsilon\,\mathrm e^{-\min F/2}\le \mathrm e^{-\min F/2},
\end{equation*}
then, \eqref{eqn:46} yields
\begin{equation}\label{eqn:48}
\epsilon\exp\Big(- \mathrm e^{-\min F/2}\, C_0\Bigr)\le\mu M\le \epsilon
\end{equation}
and
\begin{equation}\label{eqn:49}
\exp\Big(- \mathrm e^{-\min F/2}\, C_0\Bigr)\max (\Delta u+u_t+2)\le M.
\end{equation}
Eventually from   \eqref{eqn:45}, \eqref{eqn:48}, and \eqref{eqn:49} we obtain
\begin{multline*}
\epsilon\exp\Big(- 2\mathrm e^{-\min F/2}\, C_0\Bigr) \max(\Delta u+u_t+2)
\le \\ \le
\bigabs{(\Delta F+F_t)\mathrm e^F}_{C^0}+2 \epsilon\abs{\mathrm e^F}_{C^0}+
\epsilon^2\abs{\nabla  u}_{C^0}^2,
\end{multline*}
that is
\begin{multline}\label{eqn:51}
 \max(\Delta u+u_t+2)\le
\\
\le
\exp\Big(2\mathrm e^{-\min F/2}\, \abs{u}_{C^0}\Bigr)
\biggl(\frac 1\epsilon\,\bigabs{(\Delta F+F_t)\mathrm e^F}_{C^0}+2 \abs{\mathrm e^F}_{C^0}+
3 \epsilon\abs{\nabla  u}_{C^0}^2\biggr).
\end{multline}
Since
\begin{equation*}
\abs{\Delta u}_{C^0}\le  \max(\Delta u+u_t+2)+\abs{\nabla u}_{C^0}+2,
\end{equation*}
estimate \eqref{eqn:50} follows  from  \eqref{eqn:51}, with
\begin{equation*}
\epsilon=\frac 1{1+\abs{\nabla  u}_{C^0}}.\qedhere
\end{equation*}
\end{proof}

To prove next theorem, we need the following estimate.
\begin{proposition}\label{pro:12}
Given  $0<\mu<1$,  there exists a positive  $K_0$, depending only on $\mu$, such that
\begin{equation}\label{eqn:145}
\abs{u}_{C^{1+\mu}}\le K_0\Bigl(\norm{u}_{C^0}+\abs{\Delta u}_{C^0}\Bigr),
\qquad \textup{for all $u\in C^2(T^3)$}.
\end{equation}
\end{proposition}
\begin{proof}
Let $p=\frac 3{1-\mu}$. Since $p>3$, Morrey inequality gives
\begin{equation*}
\abs{u}_{C^{1+\mu}}\le C \norm{u}_{W^{2,p}},
\end{equation*}
where the constant $C$ depends only on $\mu$. On the other hand, 
elliptic $L^p$ estimates for the Laplacian give
\begin{equation*}
\norm{u}_{W^{2,p}}\le C'\bigl(\norm{u}_{L^p}+\norm{\Delta u}_{L^p}\bigr),
\end{equation*}
where again  $C'$ depends only on $\mu$.

Finally, if $u\in C^2(T^3)$ we have
\begin{equation*}
\norm{u}_{L^p}+\norm{\Delta u}_{L^p}\le \abs{u}_{C^0}+\abs{\Delta u}_{C^0}. \qedhere
\end{equation*}
\end{proof}
\begin{theorem}\label{thm:5}
Consider  $F\in C^2(T^3)$ satisfying condition \eqref{eqn:29}. Then there exists a positive constant $C_2$, depending only on
 $\norm{F}_{C^2}$, such that
\begin{equation}\label{eqn:165}
\abs{u}_{C^1}\le C_2,
\end{equation}
for all $u\in \tilde C^4(T^3)$ solution to equation~\eqref{eqn:8}.
\end{theorem}
\begin{proof}
Let $0<\mu<1$.
Thanks  to standard interpolation theory (see \cite[section 6.8]{GT}), for all  $\epsilon>0$ there exists a
positive constant $M_\epsilon$ such that
\begin{equation*}
\abs{u}_{C^1}\le M_\epsilon \abs{u}_{C^0}+\epsilon\abs{u}_{C^{1+\mu}},\qquad\textup{for all $u\in C^{1+\mu}(T^3)$}.
\end{equation*}
Then, thanks to Theorem \ref{thm:7} and Proposition \ref{pro:12}, we have
\begin{equation*}
\begin{aligned}[t] \abs{u}_{C^1}&\le M_\epsilon C_0+\epsilon K_0\Bigl(C_0+\abs{u}_{C^1}+\abs{\Delta u}_{C^0}\Bigr)
\\ &\le
M_\epsilon C_0+\epsilon K_0\Bigl(C_0+\abs{u}_{C^1}+C_1(1+\abs{u}_{C^1})\Bigr)
\\ &=
M_\epsilon C_0+\epsilon K_0(C_0+C_1)+\epsilon K_0(1+C_1)\abs{u}_{C^1},
\end{aligned}\end{equation*}
which implies  \eqref{eqn:165}, if we choose
\begin{equation*}
\epsilon<\frac 1{K_0(1+C_1)}. \qedhere
\end{equation*}
\end{proof}
\begin{corollary}\label{cor:10}
Under the hypotheses of Theorem~\textup{\ref{thm:5}},
we have that equation~\eqref{eqn:8} is  uniformly elliptic on the set $\mathcal S$ of all solutions $u\in \tilde C^4(T^3)$, in the sense that
\begin{equation*}
\inf_{u\in\mathcal S} \Lambda(u)>0,
\end{equation*}
where $\Lambda$ is defined in \eqref{eqn:61}.
\end{corollary}
\begin{proof}
It follows from Proposition~\ref{pro:5} and Theorems~\ref{thm:7} and~\ref{thm:5}.
\end{proof}

\subsection{$C^{2+\rho}$-estimate}
\rule{0pt}{0pt}

We begin by recalling a theorem of \cite{TWWY}, which greatly simplifies the estimate of derivatives up to second order. In
\cite{TWWY} the theorem has been stated locally, but on compact manifolds it holds globally.

\begin{theorem}[\protect{\cite[Theorem 5.1]{TWWY}}]\label{thm:11}
Let $\tilde \Omega$ be be the solution of the Calabi-Yau equation
\begin{equation*}
\tilde \Omega^n=\mathrm e^F\Omega^n, \qquad  [\tilde \Omega]=[\Omega],
\end{equation*}
on a compact almost-K\"ahler manifold $(M^{2n},\Omega,J)$.

Assume there are two constants $\tilde C_0>0$ and $0<\rho_0<1$ such that $F\in C^{\rho_0}(M^{2n})$ and
\begin{equation*}
\tr \tilde g\le \tilde C_0,
\end{equation*}
where $\tilde g$ is the Riemannian metric associated to $\tilde \Omega$.

Then there exist two constants $\tilde C>0$ and $0<\rho<1$, depending only on $M^{2n}$, $\Omega$, $J$, $C_0$ and
$\norm{F}_{C^{\rho_0}}$, such that
\begin{equation*}
\norm{\tilde g}_{C^\rho}\le \tilde C.
\end{equation*}
\end{theorem}

Using this  Theorem we easily prove the following estimate.
\begin{theorem}\label{thm:8}
Given  $F\in C^2(T^3)$ satisfying condition \eqref{eqn:29}, there exist  constants $C_3>0$ and $\rho>0$, both depending
only on $\norm{F}_{C^2}$, such that
\begin{equation}\label{eqn:68}
\norm{u}_{C^{2+\rho}}\le C_3,
\end{equation}
for all $u\in \tilde C^4(T^3)$ solution to  equation~\eqref{eqn:8}.
\end{theorem}
\begin{proof}
From \eqref{eqn:71} we obtain that the Riemannian metric $\tilde g$ is represented by the matrix
\begin{equation*} \tilde
g=\begin{bmatrix}
u_{yy}+u_{tt}+u_t+1 & u_{xy} & 0 & u_{xt} \\
u_{xy} & u_{xx}+1 & u_{xt} & 0 \\
0 & u_{xt} & u_{yy}+u_{tt}+u_t+1 & -u_{xy} \\
u_{xt} & 0 & -u_{xy} & u_{xx}+1
\end{bmatrix}. \end{equation*}
Then
\begin{equation*}
\tr \tilde g=2(\Delta u+u_t+2).
\end{equation*}
Thanks to Theorems \ref{thm:7} and \ref{thm:5}  we can apply Theorem~\ref{thm:11} and get that
\begin{equation}\label{eqn:86}
\max\bigl\{\norm{1+u_{xx}}_{C^\rho},\,\norm{1+u_{yy}+u_{tt}+u_t}_{C^\rho},
\,\norm{u_{xy}}_{C^\rho},\,\norm{u_{xt}}_{C^\rho}\bigr\}\le \tilde C,
\end{equation}
where $\tilde C$ depends only on $\norm{F}_{C^2}$.

Now the estimates of    second order derivatives can be obtained as follows. Given a solution $u$ of equation
\eqref{eqn:8}, we have that $u$ can be viewed as a solution to the linear PDE
\begin{equation}\label{eqn:73}
Pu_{xx}+Q(u_{yy}+u_{tt})-2Ru_{xy}-2Su_{xt}+Qu_t=f
\end{equation}
with
\begin{equation*}
P=u_{yy}+u_{tt}+u_t+1,\quad Q=u_{xx}+1,\quad R=u_{xy},\quad S=u_{xt},
\end{equation*}
and
\begin{equation*}
f=2\mathrm e^F-(\Delta u+u_t+2).
\end{equation*}
Thanks to Proposition~\ref{pro:5}, Corollary~\ref{cor:10} and estimate~\eqref{eqn:86}, standard Schauder theory gives the estimate
\eqref{eqn:68}.
\end{proof}

\section{Proof of Theorem~\ref{thm:12}}\label{sec:5}

\begin{proposition}\label{pro:11} Assume $u\in \tilde C^{2+\rho}(T^3)$ is a solution to equation
\eqref{eqn:8} with $\rho>0$. If $F\in C^\infty(T^3)$ then $u\in \tilde C^\infty(T^3)$.
\end{proposition}
\begin{proof}
From Proposition~\ref{pro:5} we have that equation
\eqref{eqn:8} is elliptic. Then from \cite[Theorem 4.8, Chapter~14]{T1},
 it follows that
$u$ belongs to the Sobolev space $W^{n,2}(T^3)$, for all $n\in\mathbb Z_+$. But this implies that $u\in C^\infty(T^3)$.
\end{proof}

Thanks to Theorem~\ref{thm:1}, Theorem~\ref{thm:12} is an immediate consequence of the following
\begin{theorem}
Let $F\in C^\infty(T^3)$ satisfy \eqref{eqn:29}. Then equation~\eqref{eqn:8} has a solution $u\in\tilde C^\infty(T^3)$.
\end{theorem}
\begin{proof}
We apply the continuity method (see \cite[Section 17.2]{GT}).
 For $0\le \tau\le 1$, let
\begin{equation}
\mathfrak S_\tau=\Bigl\{u\in \tilde C^\infty(T^3)\mid (u_{yy}+u_{tt}+u_t+1)(u_{xx}+1)-u_{xy}^2-u_{xt}^2=
\mathrm e^{F_\tau}\Bigr\}
\end{equation}
where
\begin{equation*} F_\tau=\log(1-\tau+\tau\, \mathrm e^F).
\end{equation*}
Note that $0\in \mathfrak S_0$ and that  $\mathfrak S_1$ consists in the solutions to \eqref{eqn:8} lying in $\tilde
C^{\infty}(T^3)$. Since
\begin{equation*}\max_{0\le \tau\le 1}\norm{F_\tau}_{C^2}<\infty,
\end{equation*}
 and
\begin{equation*}
\int_{T^3}\mathrm e^{F_\tau}\,  dV =
\int_{T^3}\bigl(1-\tau+\tau\,\mathrm e^{F}\bigr)\, dV=1,
\end{equation*} by Theorem~\ref{thm:8} there exists a
real number $\rho>0$ such that
\begin{equation}\label{eqn:67}
\sup_{u\in\mathfrak S}\norm{u}_{C^{2+\rho}}<\infty,
\end{equation}
with
\begin{equation*}
\mathfrak S=\bigcup_{0\le\tau\le 1}\mathfrak S_\tau\ne \emptyset.
\end{equation*}
Since $0\in \mathfrak S_0$, the set $\bigl\{\tau\in [0,1]\mid \mathfrak S_\tau\ne \emptyset\bigr\}$ is not empty and we
can define
\begin{equation*}
\mu=\sup\bigl\{\tau\in [0,1]\mid \mathfrak S_\tau\ne \emptyset\bigr\}.
\end{equation*}
In order to compete the proof we have to show that $\mathfrak S_\mu\ne \emptyset$ and $\mu=1$.

\begin{list}{$\bullet$}{\setlength{\labelwidth}{10pt}\setlength{\leftmargin}{15pt}} \item $\mathfrak S_\mu\ne \emptyset$.
By the definition of $\mu$ there exist two sequences $(\tau_k)\subset [0,1]$ and $(u_k)\subset\tilde C^\infty(T^3)$
such that $(\mu_k)$ is increasing and $u_k\in \mathfrak S_{\tau_k}$ for all $k$. Thanks to \eqref{eqn:67}, the sequence
$(u_k)$ is bounded in $\tilde C^\rho(T^3)$, then by Ascoli-Arzel\`a Theorem there exists a subsequence $(u_{k_j})$
convergent in $\tilde C^{2+\rho/2}(T^3)$. Let $v=\lim u_{k_j}$. Then $v$ belongs to $\tilde C^{2+\rho/2}(T^3)$ and
satisfies the equation
\begin{equation*}
(v_{yy}+v_{tt}+v_b+1)(v_{xx}+1)-v_{xy}^2-v_{xt}^2=\mathrm e^{F_\mu}.
\end{equation*}
By Proposition~\ref{pro:11} $v$ belongs to $\tilde C^\infty( T^3)$. In particular, $v$ belongs to $\mathfrak S_\mu$,
which turns out to be not empty.
\item $\mu=1$. Assume by contradiction $\mu<1$ and define the non-linear $C^\infty$ operator\,\footnote{\
$\int_{T^3}T(u,\tau)\,dV=0$ follows from
$\int_{T^3}(u_{xy}^2+u_{xt}^2)\,dV=\int_{T^3}(u_{yy}+u_{tt})u_{xx}\,dV$ \newline and $\int_{T^3}\mathrm e^{F_\tau}\,dV=1$.
}
\begin{equation*}
\begin{cases} T:\tilde C^\rho(T^3)\times [0,1]\to \tilde C^{\rho-2}(T^3), \\
T(u,\tau)=(u_{yy}+u_{tt}+u_t+1)(u_{xx}+1)-u_{xy}^2-u_{xt}^2-\mathrm e^{F_\tau}.
\end{cases}
\end{equation*}
Since $\mathfrak S_{\mu}$ is not empty, there exists $v\in \mathfrak S_\mu$ such that $T(v,\mu)=0$. Compute
\begin{equation*}
\partial_1 T(v,\mu)w=L w,
\end{equation*}
where
\begin{equation*}
Lw=Pw_{xx}+Q(w_{yy}+w_{tt})-2Rw_{xy}-2Sw_{xt}+Qw_t=f
\end{equation*}
with
\begin{equation*}
P=v_{yy}+v_{tt}+v_t+1,\quad Q=v_{xx}+1,\quad R=v_{xy},\quad S=v_{xt},
\end{equation*}
Since $v\in\mathfrak S_\mu$, we know that  $L:\tilde C^{2+\rho}(T^3) \to \tilde C^\rho(T^3)$ is elliptic. Then by
Strong Maximum Principle $L=0$ implies that $u$ is constant. This shows that $L$ is is one-to-one on $\tilde
C^{2+\rho}$. Moreover, by ellipticity, $L$ has closed range, thus Schauder Theory and Continuity Method (see
\cite[Theorem 5.2]{GT}) show that $L$ is onto. Therefore by Implicit Function Theorem there exists an $\epsilon>0$ such
that
\begin{equation*}
T(u,\tau)=0
\end{equation*}
is solvable with respect to $u$ for every $\tau\in\left (\mu-\epsilon,\mu+\epsilon\right )$. Thanks to Proposition
\ref{pro:11}, these solutions belong to $\tilde C^\infty(T^3)$. Then $\mathfrak S_\tau\ne \emptyset$ for all
$\mu<\tau<\mu+\epsilon$,  in contradiction with the definition of $\mu$. \qedhere
\end{list}
\end{proof}

\section{Outline of the proof of Theorem \ref{thm:10}}\label{sec:6}
Let $\theta$ as in the statement of Theorem \ref{thm:10}. Then we can write
\begin{equation*}
\omega_{\theta}=f^{13}-f^{24},
\end{equation*}
with
\begin{equation*}
f^1=\cos\theta\, e^1+\sin\theta\, e^2,\quad
f^2=-\sin\theta\, e^1+\cos\theta\, e^2,\quad f^3=e^3,\quad f^4=e^4.
\end{equation*}
Since
\begin{equation*}
df^4=de^4=e^{12}=f^{12},
\end{equation*}
one easily obtains that
\begin{equation*}
\alpha=d^{\mathrm c}u-uf^1
\end{equation*}
satisfies \eqref{eqn:88} and \eqref{eqn:103}
if and only if $u\in \tilde C^2(T^3)$ is a solution to the fully non-linear PDE
\begin{multline}\label{eqn:172}
\Bigl((\cos\theta\,\partial_x-\sin\theta\,\partial_y)^2 u+1\Bigr)
\Bigl((\sin\theta\,\partial_x+\cos\theta\,\partial_y)^2 u+\partial_t^2 u+\partial_t u+1\Bigr)-
\\-\Bigl((\cos\theta\,\partial_x-\sin\theta\,\partial_y)(\sin\theta\,\partial_x+\cos\theta\,\partial_y)u\Bigr)^2-
\\-\Bigl((\cos\theta\,\partial_x-\sin\theta\,\partial_y)\partial_t u\Bigr)^2={\mathrm e}^F.
\end{multline}

Let
\begin{equation*}
v(p,q,t)=u(x,y,t),
\end{equation*}
with
\begin{equation*}
\begin{cases} x=\cos\theta\, p+\sin\theta\, q,\\ y=-\sin\theta\,p+\cos\theta\, q.\end{cases}
\end{equation*}
Then
\begin{equation*}
\begin{cases}
\partial_p v=\cos\theta\,\partial_x u-\sin\theta\,\partial_y u, \\ \partial_q v=\sin\theta\,\partial_x u+\cos\theta\,\partial_y u.\end{cases}
\end{equation*}
This implies that \eqref{eqn:172} can be re-written as
\begin{equation}\label{eqn:173}
(v_{pp}+1)(v_{qq}+v_{tt}+v_t+1)-v_{pq}^2-v_{pt}^2={\mathrm e}^G.
\end{equation}
where
\begin{equation*}
G(p,q,t)=F(x,y,t).
\end{equation*}

Equation~\eqref{eqn:173} is formally the same as equation~\eqref{eqn:8}. There is however a big difference
in periodicity conditions, which become
\begin{equation*}
v(p+\cos\theta\, m+\sin\theta n,q+\sin\theta\, m+\cos\theta\, n,t+k)=v(p,q,t),
\end{equation*}
for all $m,n,k\in\mathbb Z$.

In particular, this implies that the proof of Proposition~\ref{pro:10} fails, unless $v$ is periodic  with respect to the first variable $p$. An elementary argument shows that this happens if and only if either $\cos\theta=0$ or $\tan\theta\in\mathbb Q$, that is if and only if there exist two integers $m$ and $n$ such that
\begin{equation*}
m^2+n^2>0
\end{equation*}
and
\begin{equation*}
\cos\theta=\frac {m}{\sqrt{m^2+n^2}},\qquad
\sin\theta=\frac {n}{\sqrt{m^2+n^2}}.
\end{equation*}
Then
\begin{equation*}
v(p+\sqrt{m^2+n^2},q,t)=v(p,q,t),
\end{equation*}
and from $v_{pp}>-1$ we get the estimate
\begin{equation*}
\abs{v_p}\le \sqrt{m^2+n^2}.
\end{equation*}
The rest of the proof of Theorem \ref{thm:10} can be obtained by a slight modification of the argument used to prove
Theorem \ref{thm:12} and it is left to the reader.

\bigskip

\footnotesize\parindent0pt\parskip8pt

Ernesto Buzano, Anna Fino and Luigi Vezzoni,\\
Dipartimento di Matematica, Universit\`a di Torino, Via
Carlo Alberto 10, 10123 Torino, Italy.\\
E-mail: \texttt{ernesto.buzano@unito.it, annamaria.fino@unito.it, luigi.vezzoni@unito.it}

\end{document}